\documentclass[10pt, reqno]{amsart}
\usepackage{graphicx, amssymb, amsmath, amsthm}
\usepackage{epsfig}
\usepackage{hyperref}
\usepackage[usenames, dvipsnames]{color}
\usepackage{color}
\numberwithin{equation}{section}

\def\ov{\overline}

\newcommand{\pe}{\perp}
\let\Re=\undefined\DeclareMathOperator*{\Re}{Re}

\newcommand{\R}{\mathbb{R}}

\newtheorem{theorem}{Theorem}[section]

\newtheorem{lemma}[theorem]{Lemma}

\newtheorem{corollary}[theorem]{Corollary}

\newtheorem{proposition}[theorem]{Proposition}

\theoremstyle{definition}

\newtheorem{remark}[theorem]{Remark}

\makeatletter
\newcommand{\Extend}[5]{\ext@arrow0099{\arrowfill@#1#2#3}{#4}{#5}}
\makeatother

\begin{document}
\title[ KP-I equation  ]{Exact controllability of linear KP-I equation}

\author[]{Chenmin Sun }
\address{Universit\'e C\^{o}te d'Azur, LJAD, France}
\email{csun@unice.fr}
\thanks{The author is supported by the European Research Council, ERC-2012-ADG, project number 320845: Semi classical Analysis and Partial Differential Equations}

 \maketitle

\begin{abstract}
We prove the exact controllability of linear KP-I equation if the control input is added on a vertical domain. More generally, we have obtained the least dispersion needed to insure observability for fractional linear KP I equation.   
\end{abstract}

\section{Introduction}
In this note, we complete the study of control problem for linear KP type equations started in \cite{Rivas-Sun}. The precise model considered here is the linear KP-I equation
\begin{equation}\label{linearKP-I}
\begin{split}
\partial_tu+\partial_x^3u-\partial_x^{-1}\partial_y^2u=0,
\end{split}
\end{equation} 
where the Fourier multiplier  $\partial_x^{-1}$ is defined by
$$ \widehat{\partial_x^{-1}v}(k,\l)=\frac{1}{ik}\widehat{v}(k,\l)
$$
for all functions
$$ v\in \mathcal{D}_0'(\mathbb{T}^2):=\{v\in \mathcal{D}_0'(\mathbb{T}^2):\widehat{v}(0,l)=0 \textrm{ for all }l\in\mathbb{Z}\}.
$$
We denote by $L_0^2(\mathbb{T}^2)=L^2(\mathbb{T}^2)\cap\mathcal{D}_0'(\mathbb{T}^2)$.
For a vertical control region of the form $\omega=(a,b)_x\times\mathbb{T}_y$, we fix a non-negative real  function $g\in C_c^2(\mathbb{T})$  with 
$\int_{\mathbb{T}}g=1$. In this case, we define the control input by
\begin{equation}\label{verticalcontrol}
\mathcal{G}(h)(x,y)=\mathcal{G}_{\pe}(h)(x,y):=g(x)\left(h(x,y)-\int_{\mathbb{T}}g(x')h(x',y)dx'\right).
\end{equation} 
The main result of this note is the observability from a vertical region.
\begin{theorem}\label{observabilitytheorem}
For any $T>0$, there exists $C_T>0$, such that for any solution $u\in C(\mathbb{R};L_0^2(\mathbb{T}^2))$ of \eqref{linearKP-I}, we have
\begin{equation}\label{observabilityinequality}
\|u(0)\|_{L^2(\mathbb{T}^2)}^2\leq C_T\int_0^T\int_{\mathbb{T}^2}|\mathcal{G}u(t,x,y)|^2dxdydt.
\end{equation}
\end{theorem}
As explained in \cite{Rivas-Sun}, a consequence of HUM method of Lions \cite{JLions} is the  exact controllability of linear KP-I equation from vertical domain. 
\begin{theorem}\label{controltheorem}
	Given any $T>0$ and $u_0\in L_0^2(\mathbb{T}^2), u_1\in L_0^2(\mathbb{T}^2)$, there exists $f\in L^2((0,T);L^2(\mathbb{T}^2)$ such that the solution of the equation
	\begin{equation}\label{control}
	\partial_tu+\partial_x^3u-\partial_x^{-1}\partial_y^2u=\mathcal{G}f, u|_{t=0}=u_0
	\end{equation}
	satisfies $u|_{t=T}=u_1$.
\end{theorem}
\begin{remark}
	When the control region is a horizontal strip of the form $\omega=\mathbb{T}_x\times (a,b)_y$ and we put the control input as
	\begin{equation}
	\mathcal{G}_{\parallel}h(x,y):=g(y)\left(h(x,y)-\int_{\mathbb{T}}g(y')h(x,y')dy'\right).
	\end{equation}
	Then for any given time $T>0$, the similar observability \eqref{observabilityinequality}, replacing $\mathcal{G}u$ by $\mathcal{G}_{\parallel}u$, does not hold true. The argument is the same as the treatment for linear KP-II in \cite{Rivas-Sun}.
\end{remark}
The main part of this note is devoted to the proof of Theorem \ref{controltheorem}. In appendix, we discuss the validity of the observability for fractional linear KP I of the form

\begin{equation}\label{fKPIintro} \partial_tu-|D_x|^{\alpha}\partial_xu-\partial_x^{-1}\partial_y^2u=0.
\end{equation}
We will prove the following dichotomy result which asserts the least dispersion needed for the observability.
\begin{theorem}\label{dichotomy}
\begin{enumerate}
	\item If $\alpha\geq 1$, then for any $T>0$, there exists $C_T>0$, such that 
	$$ \|u(0)\|_{L^2(\mathbb{T}^2)}^2\leq C_T\int_0^T\int_{\mathbb{T}^2}|\mathcal{G}u(t,x,y)|^2dxdydt
	$$
	holds for any solution $u$ of \eqref{fKPIintro}.
	\item If $0<\alpha<1$, then for any $T>0$, there exist a sequence of solution $(u_n)$ of \eqref{fKPIintro}, such that
	$$ \displaystyle{\lim_{n\rightarrow\infty}\frac{\int_0^T\int_{\mathbb{T}^2}|\mathcal{G}u_n(t,x,y)|^2dtdxdy}{\|u_n(0)\|_{L^2(\mathbb{T}^2)}^2}}=0.
	$$ 
\end{enumerate}
\end{theorem}

\section{Notations and Preliminaries}

\subsection{Notation}
We identify $\mathbb{T}^d=\mathbb{R}^d\slash(2\pi\mathbb{Z}^d)$ with fundamental domain $[-\pi,\pi]^d$. The Fourier transform on $\mathbb{T}^d$ is denoted by
$$ \widehat{f}(\xi)=(2\pi)^{-d}\int_{\mathbb{T}^d}f(z)e^{-iz\cdot\xi}dz, \quad \xi\in\mathbb{Z}^d.
$$ 
In the case where there is no risk of confusing, we will also use $\widehat{f}$ to note the Fourier transform of one variable. For the derivative, we sometimes use the notation $D_t=\frac{1}{i}\partial_t,D_x=\frac{1}{i}\partial_{x}$.

We will only use $L^2$ based norms for this linear problem, hence we denote by
$$ \|v\|:=\|v\|_{L^2(\mathbb{T}^d)},\|v\|_s:=\|v\|_{H^s(\mathbb{T}^d)},\|f\|_T:=\|f\|_{L^2(0,T;L^2(\mathbb{T}^d))}.
$$
We will also use the inner product notations
$$  (u,v):=\int_{\mathbb{T}^d}u(x)\cdot\ov{v}(x)dx,\quad (f,w)_T:=\int_0^T\int_{\mathbb{T}^d}f(t,x)\ov{w}(t,x)dxdt,
$$
where $d=1$ or $2$, which will be clear in the context.

 \subsubsection{Symbols and quantization on Torus}
 We briefly review the $h$ pseudo-differential calculus on Torus. For $m\in\mathbb{R}$, let $S^m$ be the set of $h$-dependent functions $a(x,\xi,h)$ with parameter $h\in(0,1)$ such that for any indices $\alpha,\beta$,
 \begin{equation}\label{decayofsymbol}
 \sup_{(x,\xi,h)\in\mathbb{R}^{2d}\times(0,1)}|\partial_x^{\alpha}\partial_{\xi}^{\beta}a(x,\xi,h)|\leq C_{\alpha,\beta}(1+|\xi|)^{m-|\beta|}.
 \end{equation} 
 
 For $a\in S^m$, we denote by $\mathrm{Op}_{h}(a)$ the $h$ pseudo-differential operator acting on Schwartz functions via
 $$ \mathrm{Op}_{h}(a)f(x):=\frac{1}{(2\pi h)^d}\int_{\mathbb{R}^{2d}}e^{\frac{i(x-y)\cdot\xi}{h}}a(x,\xi,h)f(y)dyd\xi.
 $$
 We refer \cite{booksemiclassical} for symbolic calculus and another basic properties about $h$ pseudo-differential operator. For functions on a compact Riemannian manifold,  we can define $h$ pseudo-differential operator by using local coordinate and partition of unity. On the torus, we can also use the global definition of pseudo-differential calculus. Denote by $S_{per}^m$ be symbols in $S^m(\mathbb{R}^{2d})$ which are $2\pi$-periodic in $x\in\mathbb{R}^d$, namely
 $$ a(x+2\pi k,\xi)=a(x,\xi),\quad\forall (x,\xi)\in\mathbb{R}^{2d},k\in\mathbb{Z}^d.
 $$  
 Symbols in $S_{per}^m$ can depend on $h$ with uniform estimate \eqref{decayofsymbol},though the dependence is not displayed in our notation.  
 We quantize $a\in S_{per}^m$ as an operator on $\mathcal{S}'(\mathbb{T}^d)$ via the formula
 \begin{equation}\label{quantization1}
 \mathrm{Op}_h(a)f(x):=\sum_{k\in\mathbb{Z}^d}\frac{1}{(2\pi h)^d}\int_{[-\pi,\pi]^d}\int_{\mathbb{R}^d}a(x,\xi)e^{\frac{i(x-y+2k\pi)\cdot\xi}{h}}f(y)dyd\xi
 \end{equation}
 From Poisson summation formula, we have
 \begin{equation}\label{quantization2}
 \mathrm{Op}_h(a)f(x)=\sum_{k\in\mathbb{Z}^d}a(x,h_1k)\widehat{f}(k)e^{ik\cdot x}.
 \end{equation}
 The globally defined quantization via \eqref{quantization2} is equivalence to (modulo $hS^{m-1}$) the usual definition via partition of unity, see the exercise in the book \cite{SA-PG}.

\subsection{Quick review of 1D semi-classical reduction}

	Expanding the solution $u(t,x,y)$ to \eqref{linearKP-I} in Fourier series in $y$ variable
	$$ u(t,x,y)=\sum_{l\in\mathbb{Z}}a_{l}(t,x)e^{ily},
	$$
	we find that for each $l\in\mathbb{Z}$, $a_l$ satisfies the equation
	$$ \partial_ta_l+\partial_x^3a_l-l^2\partial_x^{-1}a_l=0
	$$
	Therefore, by changing notations, the equation \eqref{linearKP-I} can be reduced to the study of the following $\lambda$ dependent equation
	\begin{align} \label{1Dequation}
	\begin{cases}    \partial_tu+\partial_x^3u+\lambda^2\partial_x^{-1}u=0,\quad
	(t,x)\in\R\times\mathbb{T},
	\\
	u|_{t=0}=u_0\in L_0^2(\mathbb{T}),
	\end{cases}
	\end{align}
We take $\lambda=\frac{1}{h^2}$ and rewrite \eqref{1Dequation} as
\begin{align} \label{1Dsemiequation}
\begin{cases}    h^3D_tu-(hD_x)^3u-(hD_x)^{-1}u=0,\quad
(t,x)\in\R\times\mathbb{T},
\\
u|_{t=0}=u_0\in L_0^2(\mathbb{T}),
\end{cases}
\end{align}
The solution $u$ depends on the parameter $h$ and we will drop the dependence in the sequel. From the same proof of Proposition 3.5 in \cite{Rivas-Sun}, we reduce the proof of Theorem \ref{observabilitytheorem} to the following weak observability.
\begin{theorem}\label{reduction1}
$T>0$ be given. There exist a constant $C_T>0$ and a sufficiently small number $h_0>0$, such that for all $h\in(0,h_0)$, the solution $u$ of the $h$ dependent equation \eqref{1Dsemiequation} satisfies
	\begin{equation}\label{observabilityuniform1}
	\|u_0\|^2\leq C_T\int_0^T\int_{\mathbb{T}}|g(x)u(t,x)|^2dxdt+C_T\|u_0\|_{-1}^2.
	\end{equation}
\end{theorem}
We use a standard homogeneous Littlewood-Paley decomposition. Take $\psi\in C_c^{\infty}(\mathbb{R})$ with support supp$\psi\subset \{1/2\leq |\xi|\leq 2\}$ and $\psi_n\in C_c^{\infty}(\mathbb{R})$ such that
$$ \sum_{n\in\mathbb{Z}}\psi_n(\xi)=1,\quad\forall\xi\neq 0,
$$
where $\psi_n(\xi)=\psi(2^{n}\xi)$. With this notation, 
we further reduce the proof of Theorem \ref{reduction1} to the following frequency-localized estimate.
\begin{proposition}\label{reduction2}
	Let $T>0$ and $\epsilon_0>0$ be given. There exist $h_0>0,$ small and $C_0=C_0(\epsilon_0,T)>0$ such that for all $n\in\mathbb{Z}$ which subject to $2^nh\leq\epsilon_0$,
	\begin{equation}\label{singlefrequency}
	\|\psi_n(hD_x)u(0)\|^2\leq C_0\int_0^{T_0}\int_{\mathbb{T}}|g(x)\psi_n(hD_x)u(t,x)|^2dxdt
	\end{equation}
	holds true for all solutions $u(t,x)$ of \eqref{1Dsemiequation}.
\end{proposition}
The derivation from Proposition \ref{reduction2} to Theorem \ref{reduction1} is simple and can be found in  \cite{Rivas-Sun}.
The remaining part of this note is devoted to the proof of \eqref{singlefrequency}. We summarize the path of the proof as follows:
\begin{itemize}
\item Regimes $n\geq N_0$ and $n\leq -N_0$: $n\leq -N_0$ corresponds to the very low frequency regime in which the term $(hD_x)^{-1}$ dominates the dispersion. $n\geq N_0$ corresponds to the very high frequency regime in which the term $(hD_x)^{3}$ dominates the dispersion. The arguments are similar as for linear KP-II.
\item Regime $|n|\leq N_0$: This is the essential  difference between KP-I and KP-II. The group velocity of KP-I could be very small in this regime. 
\end{itemize}

\section{The proof of Proposition \ref{reduction2}}

\subsection{Regimes far from critical points}
Let us consider the following $\epsilon-$dependence symbols:
$$p_{\epsilon}(x,\xi)=\left(\frac{\epsilon^4}{\xi}+\xi^3\right)\chi(\xi),\quad q_{\epsilon}(x,\xi)=\left(\frac{1}{\xi}+\epsilon^4\xi^3\right)\chi(\xi),
$$
where $\chi\in C_c^{\infty}(\mathbb{R})$ with supp$(\chi)\subset\{\mu<|\xi|<\nu\}$ for some $0<\mu<\frac{1}{2}$, $\nu>2$ and $\chi\equiv 1$ in a neighborhood of $\{1/2\leq|\xi|\leq 2\}$. Denote by $P_{\epsilon}=\mathrm{Op}_{\widetilde{h}}(p_{\epsilon})$ and $Q_{\epsilon}=\mathrm{Op}_{\widetilde{h}}(q_{\epsilon})$.We use the notations $U_{\epsilon}(t),V_{\epsilon}(t)$  to represent solution operators to the following two equations
\begin{align} \label{solutionoperator1}
\begin{cases}    \frac{\widetilde{h}}{i}\partial_tU_{\epsilon}(t)+U_{\epsilon}(t)P_{\epsilon}=0,
\\
U_{\epsilon}(0)=I,
\end{cases}
\end{align}
\begin{align} \label{solutionoperator2}
\begin{cases}    \frac{\widetilde{h}}{i}\partial_tV_{\epsilon}(t)+V_{\epsilon}(t)Q_{\epsilon}=0,
\\
V_{\epsilon}(0)=I
\end{cases}
\end{align}
The flows associated to the vector fields $H_{p_{\epsilon}},H_{q_{\epsilon}}$ are explicitly given by
$$ \phi_{\epsilon,t}(x_0,\xi_0)=\left(x_0+\left(
-\frac{\epsilon^4}{\xi_0^2}+3\xi_0^2\right)\chi(\xi_0)t+\left(\frac{\epsilon^4}{\xi_0}+\xi_0^3\right)\chi'(\xi_0)t,\xi_0\right),
$$
$$\varphi_{\epsilon,t}(x_0,\xi_0)=\left(x_0+\left(-\frac{1}{\xi_0^2}+3\epsilon^4\xi_0^2\right)\chi(\xi_0)t+\left(\frac{1}{\xi_0}+\epsilon^4\xi_0^3\right)\chi'(\xi_0)t,\xi_0\right)
$$
with respectively.

From Egorov's theorem (see \cite{booksemiclassical}), we know that for any symbol $a(x,\xi)\in C_c^{\infty}(T^*M)$, 
$$ U_{\epsilon}(-t)\mathrm{Op}_{\widetilde{h}}(a)U_{\epsilon}(t)=\mathrm{Op}_{\widetilde{h}}(a\circ\phi_{\epsilon,t})+O_{L^2\rightarrow L^2}(\widetilde{h}),
$$
$$V_{\epsilon}(-t)\mathrm{Op}_{\widetilde{h}}(a)V_{\epsilon}(t)=\mathrm{Op}_{\widetilde{h}}(a\circ\varphi_{\epsilon,t})+O_{L^2\rightarrow L^2}(\widetilde{h}).
$$
We remark that the bound $O_{L^2\rightarrow L^2}(\widetilde{h})$ is independent of $\epsilon\leq 1$ since all the semi-norms of the symbol $p_{\epsilon},q_{\epsilon}$ can be chosen continuously depending on $\epsilon$. \\

Now we prove the following localized observability estimates:
\begin{proposition}\label{highfrequency}
	There exists $C_0>0,T_0>0,\widetilde{h}_0>0$ and $\delta_0>0$ such that for all $u_0\in L_0^2(\mathbb{T})$, and all $\widetilde{h}\leq  \widetilde{h}_0$
	\begin{equation}\label{veryhigh}
	\|\psi(\widetilde{h}D_x)u_0\|^2\leq C_0\int_0^{T_0}\|gU_{\epsilon}(t)\psi(\widetilde{h}D_x)u_0\|^2dt,
	\end{equation}
	\begin{equation}\label{lesshigh}
	\|\psi(\widetilde{h}D_x)u_0\|^2\leq C_0\int_0^{T_0}\|gV_{\epsilon}(t)\psi(\widetilde{h}D_x)u_0\|^2dt,
	\end{equation}
uniformly in $\epsilon<\delta_0$.
\end{proposition}
\begin{proof}
	Here we only prove the first inequality, and the second one will follow in the same manner.
	Consider the symbol $a(x,\xi)=g(x)^2\widetilde{\psi}(\xi)$ (strictly speaking, $g$ is not smooth and we need approximate it by smoothing functions) and its quantization $\mathrm{Op}_{\widetilde{h}}(a)=(g(x))^2\widetilde{\psi}(\widetilde{h}D_x)$, where $\widetilde{\psi}$ is a slight enlargement of $\psi$ so that $\widetilde{\psi}\psi=\psi$ and supp$\chi|_{\mathrm{supp }(\widetilde{\psi})}=1$. From Egorov's theorem, we have
	$$ U_{\epsilon}(-t)\mathrm{Op}_{\widetilde{h}}(a)U_{\epsilon}(t)=\mathrm{Op}_{\widetilde{h}}(a\circ\phi_{\epsilon,t})+O_{L^2\rightarrow L^2}(\widetilde{h}),\textrm{ uniformly in } \epsilon\leq 1.
	$$
	Note that on the support of $a$, $\chi'(\xi)=0$, and thus we have $$\varphi_{\epsilon,t}(x_0,\xi_0)=\left(x_0+\left(-\frac{\epsilon^4}{\xi_0^2}+3\xi_0^2\right)t,\xi_0\right).
	$$
We choose $\delta_0=\delta_0(\mu,\nu)$, sufficiently small, such that 
 $\left|-\frac{\epsilon^4}{\xi_0^2}+3\xi_0^2\right|\geq c_0>0$, uniformly in $\epsilon<\delta_0$ on the $\xi-$support of $\widetilde{\psi}$. Therefore, for some $T_0=T_0(c_0)>0$, and $c_1>0$ , we have
	$$ \int_0^{T_0}a\circ\phi_{\epsilon,t}dt\geq c_1>0.
	$$
	Now we calculate
	\begin{equation*}
	\begin{split}
	&\int_0^{T_0}\|gU_{\epsilon}(t)\psi(\widetilde{h}D_x)u_0\|^2dt\\=&\int_0^{T_0}\left(gU_{\epsilon}(t)\psi(\widetilde{h}D_x)u_0,gU_{\epsilon}(t)\widetilde{\psi}(\widetilde{h}D_x)\psi(\widetilde{h}D_x)u_0\right)dt\\
	=&\int_0^{T_0}\left(U_{\epsilon}(-t)\widetilde{\psi}(\widetilde{h}D_x)g^2U_{\epsilon}(t)u_0,\psi(\widetilde{h}D_x)u_0\right)dt\\
	=&\left(\mathrm{Op}_{\widetilde{h}}(b_{T_0})\psi(\widetilde{h}D_x)u_0,\psi(\widetilde{h}D_x)u_0\right),
	\end{split}
	\end{equation*}
	with $b_{T_0}(x,\xi)=\int_0^{T_0}a\circ\phi_{\epsilon,t}dt$ modulo $\widetilde{h}S^0$. Thus, from Sharp G\aa rding inequality, 
	$$\left(\mathrm{Op}_{\widetilde{h}}(b_{T_0})\psi(\widetilde{h}D_x)u_0,\psi(\widetilde{h}D_x)u_0\right)\geq \frac{c_1}{2}\|\psi(\widetilde{h}D_x)u_0\|^2-C\widetilde{h}\|\psi(\widetilde{h}D_x)u_0\|^2.
	$$
	To conclude the proof, we just need choose $\widetilde{h}_0<\min\{\frac{c_1}{4C},1\}$.
\end{proof}
As a consequence, we can proof Proposition \ref{reduction2} in the easy regimes:
\begin{corollary}
There exist $h_0>0$ and a integer $N_0$ which depends on $\delta_0$ in Proposition \ref{highfrequency}, such that for all $h<h_0$, $|n|\geq N_0,$ and $ 2^nh\leq\epsilon_0$, the inequality \eqref{singlefrequency} holds true.
\end{corollary}	
\begin{proof}
Take $N_0\in\mathbb{N}$ such that $2^{-N_0}<2^{-10}\delta_0$. Fix $h_0<\min\{2^{-N_0}\epsilon_0,\widetilde{h_0}\}>0$. We first consider the case $n\leq -N_0$.
 Define a new semi-classical parameter $\widetilde{h}_n=2^nh\ll 1$ and rescale the time variable by setting $w_n(t,x):=\psi(\widetilde{h}_nD_x)u(2^{2n}t,x)$.  $w_n$ satisfies the following equation:
	$$ \widetilde{h_n}\partial_tw_n+(\widetilde{h_n}\partial_x)^3w_n+2^{4n}(\widetilde{h_n}\partial_x)^{-1}w_n=0.
	$$
	Applying  \eqref{veryhigh} to $w_n$ with $\epsilon=2^n\leq \delta_0$ and $\widetilde{h}=\widetilde{h}_n$ we obtain that
	$$ \|w_n(0)\|^2\leq C\int_0^{T_0}\|gw_n(t)\|^2dt.
	$$
	From conservation of $L^2$ norm along the flow, we apply the inequality above $2^{-2n}-1$ times to obtain
	\begin{equation*}
	\begin{split}
	\frac{1}{2^{2n}}\|u_n(0)\|^2\leq  \frac{C}{2^{2n}}\sum_{M=0}^{2^{-2n}-1}\int_{M2^{2n}T_0}^{(M+1)2^{2n}T_0}\|gu_k(t)\|^2dt
	= \frac{C}{2^{2n}}\int_0^{T_0}\|gu_n(t)\|^2dt,
	\end{split}
	\end{equation*}
	and this is exactly
	$$ \|\psi_n(hD_x)u(0)\|^2\leq C\int_0^{T_0}\|g\psi_n(hD_x)u(t)\|^2dt.
	$$
	
We next consider the case $n\geq N_0$ and $2^nh\leq\epsilon_0$. Define the new small semi-classical parameter $\widetilde{h}_n=2^{n}h$, thanks to the restriction that $2^{n}h\leq \epsilon_0\ll 1$.
	Denote by $u_n=\psi(\widetilde{h}_nD_x)u$ and define $v_n(t,x)=u_n(2^{-2k}t,x)$. Thus $v_n$ solves the equation
	$$ \widetilde{h}_n\partial_tv_n+2^{-4n}(\widetilde{h}_n\partial_x)^3v_n+(\widetilde{h}_n\partial_x)^{-1}v_n=0.
	$$
	Applying \eqref{lesshigh} with $\widetilde{h}=\widetilde{h}_n,\epsilon=2^{-n}$, we obtain that
	$$ \|v_n(0)\|^2\leq C\int_0^{T_0}\|gv_n(t)\|^2dt.
	$$
	
	Again from conservation of $L^2$ norm as in the previous argument, we finally have
	$$ \|u_n(0)\|^2\leq C\int_0^{T_0}\|gu_n(t)\|^2dt.
	$$
	
\end{proof}

\subsection{Near the critical points}
Now we prove inequality \eqref{singlefrequency} for $|n|\leq N_0$. Since $N_0$ only depends on $\mu,\nu>0$ which is chosen in a priori, it would be suffices to prove the inequality for $n=0$ only. Rewriting \eqref{1Dsemiequation} as
$$ h^3D_tu-\phi(hD_x)u=0, u=\psi(hD_x)u,
$$
with Fourier multiplier $\phi(\xi)=\xi^3+\frac{1}{\xi}$. There are only two zeros of $\phi'(\xi)=3\xi^2-\frac{1}{\xi^2}$, say $\xi_0=\pm\frac{1}{\sqrt[4]{3}}$. Splitting $\psi(\xi)=\psi_+(\xi)+\psi_-(\xi)$ with
$\psi_+=\psi\mathbf{1}_{\xi>0},$ and $ \psi_-=\psi\mathbf{1}_{\xi<0}$, it would be sufficient to prove \eqref{singlefrequency} for $u=\psi_+(hD_x)u$. 
For $\delta>0$, we take another cut-off $\chi_{\delta}\in C_c^{\infty}(\mathbb{R})$ such that
$$ \chi_{\delta}(\xi)|_{|\xi-\xi_0|\leq \delta }\equiv 1, \quad  \chi_{\delta}(\xi)|_{|\xi-\xi_0|>2\delta}\equiv 0.
$$
Taking $\delta>0$ smaller, we may assume that $\chi_{\delta}(\xi)\psi(\xi)=\chi_{\delta}(\xi).$ On the support of $(1-\chi_{\delta})\psi_+$, we have $|\psi'(\xi)|\geq c_{\delta}>0$, thus the same propagation argument as in the previous subsection yields
\begin{equation}\label{outsidecritical}
\|(1-\chi_{\delta}(hD_x))\psi_+(hD_x)u(0)\|^2\leq C_{\delta}\int_0^{T_{\delta}}\|g(1-\chi_{\delta}(hD_x))\psi_+(hD_x)u(t)\|^2dt
\end{equation}
for some $C_{\delta},T_{\delta}$ depending on $\delta>0.$ To complete the proof, it remains to prove the similar inequality for $\chi_{\delta}(hD_x)u$. Indeed, the sum of the two frequency pieces on the right hand side can be bounded by $\|g\psi_+(hD_x)u\|_{T_{\delta}}^2+C_{\delta}Th\|u(0)\|^2$ in which the error term comes from the commutator $[g,\chi_{\delta}(hD_x)]$.

Before treating the term $\chi_{\delta}(hD_x)u$, we make a further simplification. Denote by $v=\chi_{\delta}(hD_x)u$, $v=e^{i\lfloor\frac{\xi_0}{h}\rfloor x}w$, and then 
$ \displaystyle{\widehat{w}(k)=\widehat{v}\left(k+\left\lfloor\frac{\xi_0}{h}\right\rfloor\right)}.$
We see that
\begin{equation}\label{toymodel}
h^3D_tw-\Phi(hD_x)w=0, w(0)=\theta_{\delta}(hD_x)w_0,
\end{equation} 
with 
$$ \Phi(\xi)=\phi\left(\xi+h\left\lfloor\frac{\xi_0}{h}\right\rfloor\right),\theta_{\delta}(\xi)=\chi\left(\xi+h\left\lfloor\frac{\xi_0}{h}\right\rfloor\right).
$$
Note that the support of $\theta_{\delta}$ is now near the origin and 
$$
\phi'(\sigma_h)=0,\Phi''(\sigma_h)=\frac{12}{\sqrt[4]{3}}>0,\sigma_h=h\left(
\frac{\xi_0}{h}-\left\lfloor\frac{\xi_0}{h}\right\rfloor\right).$$
We are now ready to close the demonstration of Proposition \ref{reduction2} by proving the following, which is the main ingredient of this note:
\begin{proposition}\label{final}
	There exist constants $\delta>0,h_0>0$ small and $C_T>0$ such that for all $0<h<h_0$ and  $w=\theta_{\delta}(hD_x)w$, $h$ dependent solution to \eqref{toymodel}, we have
	$$ \|w(0)\|^2\leq C_T\int_0^T\int_{\mathbb{T}}\|g(x)w(t)\|^2dt.
	$$ 
\end{proposition}
The proof is down by splitting the frequency into two parts. One part contains cluster of relatively low frequencies and we control it by spectral inequality. The other part contains relatively high frequencies and will be controlled by propagation estimate after rescaling the time variable. First we notice that the inequality would not change if we replace $w$ by $w\exp\left(\frac{i\Phi(\sigma_h)t}{h^3}\right)$. We may assume that $\Phi(\sigma_h)=0$. Denote by $\displaystyle{r_h=h^{-1}\sigma_h\in [0,1)}$. For any $n_0\in\mathbb{N}$, we define the sharp frequency truncation
$$ \Pi_{\geq n_0}f:=\sum_{|k|\geq n_0}\widehat{f}(k)e^{ikx}.
$$
We divide the proof into several lemmas.
\begin{lemma}\label{highfrequencies}
Given $T>0$, there exist $N_0\in\mathbb{N}$, $h_0>0$, $C_T>0$, such that for any integer $n_0\geq N_0,h<h_0$ and $T>0$,  
\begin{equation}\label{highfrequencycontrol}
\|\Pi_{\geq n_0}w(0)\|^2\leq C_T\int_{0}^{T}\|g(x)w(t)\|^2dt
\end{equation}
holds true for all solutions of \eqref{toymodel}.
\end{lemma}
\begin{proof}
We rewrite 
$$ \Pi_{\geq n_0}w(0)=\sum_{l_0\leq l\leq L_0}\psi(2^lhD_x)\Pi_{\geq n_0}w(0),
$$
with $\displaystyle{2^{-l_0}=4\delta, 2^{-L_0}=\frac{n_0h}{4}}$. From almost orthogonal inequality
$$ \|\Pi_{\geq n_0}w(0)\|^2\leq 4\sum_{l\leq l\leq L_0}\|\psi(2^lhD_x)w(0)\|^2,
$$
we need estimate each term in the summation. By choosing $n_0\geq N_0$ large, we denote by $h_1=2^lh\leq \frac{4}{n_0}$ a new semi-classical parameter. We put $\omega=\psi(h_1D_x)w$, and then $\omega$ solves
$$ 2^{-l}h_1^3D_t\omega-\Phi_l(h_1D_x)\omega=0,
$$
with $\Phi_l(\xi)=2^{2l}\Phi(2^{-l}\xi)$. Note that $\Phi_l$ is a symbol with uniform bound in $l$ for $|\xi|\leq 2$ as well as all of its derivatives. We rescale the time by setting $v(t,x):=\omega(2^{-l}h_1^2t,x)$, hence
$$ h_1D_tv-\Phi_l(h_1D_x)v=0.
$$
Notice that $|\partial_{\xi}\Phi_l(\xi)|=|2^l\Phi'(2^{-l}\xi)|\sim 1$ for $\xi\in\mathrm{ supp }\psi$.
From the same argument as in the proof of Proposition \ref{highfrequency}, there exist $T_0>0$ and $C_{T_0}>0$ such that $$ \|\psi(h_1D_x)v(0)\|_{L^2}^2\leq C_{T_0}\int_0^{T_0}\int_{\mathbb{T}}|g(x)\psi(h_1D_x)v(t,x)|^2dxdt
	$$
	holds true for all $h_1=2^{l}h$, provided that $h<h_0$ is small enough and $n_0\geq N_0$ is large(while keeping the relation $hn_0\ll 1$).
Back to the function $w$, we have
\begin{equation*}
\begin{split}
\|\psi(h_1D_x)w(0)\|_{L^2}^2=\|\psi(h_1D_x)v(0)\|_{L^2}^2\leq &C_{T_0}\int_0^{T_0}\int_{\mathbb{T}}|g(x)\psi(h_1D_x)v(t,x)|^2dxdt\\
=&\frac{C_{T_0}}{h_1h}\int_0^{h_1hT_0}\int_{\mathbb{T}}|g(x)\psi(h_1D_x)w(t,x)|^2dxdt.
\end{split}
\end{equation*}
Thanks to the conservation of $L^2$ norm, we have for all $q\in\mathbb{N}$,
$$\|\psi(h_1D_x)w(0)\|_{L^2}^2=\|\psi(h_1D_x)w(qh_1hT_0)\|_{L^2}^2\leq \frac{C_{T_0}}{h_1h}\int_{qh_1hT_0}^{(q+1)h_1hT_0}\int_{\mathbb{T}}|g(x)\psi(h_1D_x)w(t,x)|^2dxdt.
$$
Summing $q$ from $0$ to $\lfloor \frac{T}{h_1hT_0}\rfloor$, we have
\begin{equation}\label{h_1semi}
\|\psi(h_1D_x)w(0)\|_{L^2}^2\leq C_T\int_0^T\int_{\mathbb{T}}|g(x)\psi(h_1D_x)w(t,x)|^2dxdt.
\end{equation}
Thus
\begin{equation*}
\begin{split}
\sum_{l_0\leq l\leq L_0}\|\psi(2^lhD_x)w(0)\|_{L^2}^2\leq &C_T\sum_{l_0\leq l\leq L_0}\int_0^T\int_{\mathbb{T}}|g(x)\psi(2^lhD_x)w(t,x)|^2dxdt\\
\leq &\sum_{l_0\leq l\leq L_0}C_T\int_0^T\int_{\mathbb{T}}\left(|\psi(2^lhD_x)(gw)(t,x)|^2+2^{2l}h^2|w(t,x)|^2\right)dxdt\\
\leq &C_T\int_0^T\int_{\mathbb{T}}|gw(t,x)|^2dxdt+\frac{C_T}{n_0^2}\|w(0)\|_{L^2}^2,
\end{split}
\end{equation*}
where we have used the simple commutator estimate $\|[g,\psi(2^lhD_x)]\|_{L^2\rightarrow L^2}\leq C2^lh$ in the second inequality. This completes the proof by choosing $n_0$ sufficiently large.
\end{proof}

We need the following spectral inequality, and the proof is classical and can be found in \cite{Lebeaucour}.
\begin{lemma}
	There exists an positive increasing function $\kappa:\mathbb{R}_+\rightarrow\mathbb{R}_+$, such that
	\begin{equation}\label{spectral}
	\sum_{|k|\leq m_0}|c_k|^2\leq \kappa(m_0)\int_{\mathbb{T}}\left|g(x)\sum_{|k|\leq m_0}c_ke^{ikx}\right|^2dx.
	\end{equation}
\end{lemma}
The following elementary lemma is needed in the final argument.
\begin{lemma}
For any $r\in[0,1)$, there exist $\mu_1,\mu_2\in\left[\frac{1}{8},\frac{7}{8}\right]$, such that
\begin{equation}\label{modulo}
 \mu_1+\mu_2=2r  \textrm{ mod  }\mathbb{Z}.
\end{equation}
\end{lemma}
\begin{proof}
We denote by $\{x\}:=x-\lfloor x\rfloor$, the fractional part of a real number. If $\frac{1}{4}\leq\{2r\}\leq\frac{3}{4},$ then there exist $\mu_1,\mu_2\in\left[\frac{1}{8},\frac{7}{8}\right]$, such that \eqref{modulo} holds true. If $\frac{3}{4}<\{2r\}<1$, we can choose $\mu_1=\mu_2\in\left(\frac{3}{8},\frac{1}{2}\right)$, and then $2\mu=2r$ mod $\mathbb{Z}$. For $0\leq \{2r\}<\frac{1}{4}$, we choose $\mu_1=\mu_2=\mu\in\left[\frac{1}{2},\frac{5}{8}\right)$.	
\end{proof}
Now we finish the proof.
\begin{proof}[Proof of Proposition \ref{final}]
Let $h_0,N_0$ as in Lemma \ref{highfrequencies}. Fix $h>0$ and $n_0>2N_0$ while keeping $hn_0^3\ll 1$ and $\kappa(n_0+2)h\ll 1$. From Lemma \ref{modulo}, for any $n_0\in\mathbb{N}$ satisfies $hn_0^3\ll 1$, there exist $\mu_1,\mu_2\in\left[\frac{1}{8},\frac{7}{8}\right]$ and $m_0\in-\mathbb{N}$, such that
\begin{equation}\label{relation}
(n_0+\mu_2)+(m_0+\mu_1)=2r_h.
\end{equation} 
Put $M_0:=h^{-1}\sqrt{\Phi(h(n_0+\mu_2))}$. Recall that $\Phi$ is strictly increasing for $\xi\in[\sigma_h,\delta)$ and strictly decreasing for $\xi\in (-\delta,\sigma_h]$. Thus there exist $\xi^*<r_h$, such that $\Phi(\xi^*)=h^2M_0^2$. We fix $\delta>0$ small such that $2<\Phi''(\xi)<12$ for all $|\xi|\leq\delta$. Thus $\frac{1}{4}\leq \frac{M_0^2}{(n_0+\mu_2)^2}\leq  9$. We claim that for sufficiently small $h>0$, we have $$\left\lfloor\frac{\xi^*}{h}\right\rfloor=m_0, \quad \left|\left\{\frac{\xi^*}{h}\right\}-\mu_1\right|<\frac{1}{16}. 
$$
Indeed, Taylor expansion gives
\begin{equation*}
\frac{\Phi''(\sigma_h)h^2}{2}\left(\frac{\xi^*}{h}-r_h\right)^2=\frac{\Phi''(\sigma_h)h^2}{2}\left(n_0+\mu_2-r_h\right)^2+O(h^3m_0^3),
\end{equation*}
with implicit constant in big $O$ depending only on $\displaystyle{\sup_{\xi\in(-\delta,\delta)}|\Phi'''(\xi)|}$.  As a consequence, we have
\begin{equation}\label{error}
\left|m_0+\mu_1-\frac{\xi^*}{h}\right|=O(hM_0^2).
\end{equation} 
The claim follows easily by choosing $h$ small enough.

 Define a slightly different frequency truncation
$$ w_{L}(t,x)=\sum_{k:|\Phi(hk)|\leq h^2M_0^2}\widehat{w_0}(k)e^{ikx+\frac{it\Phi(hk)}{h^3}}, \quad w_{H}=w-w_{L}.
$$
From  \eqref{spectral} and the property of $\Phi$, we have
\begin{equation}\label{low}
\|w_L(0)\|^2\leq \kappa(n_0+2)T\int_0^T\int_{\mathbb{T}}|g(x)w_L(t,x)|^2dxdt.
\end{equation}
Note that $\Pi_{\geq \frac{n_0}{2}}w_H=w_H$, then from Lemma \ref{highfrequency} we have
\begin{equation}\label{high}
\|w_H(0)\|^2\leq C_T\int_0^T\|g(x)w(t)\|^2dt.
\end{equation}
We next calculate
\begin{equation*}
\begin{split}
&\left|\int_0^T\int_{\mathbb{T}}g(x)w_H(t,x)\cdot\ov{g(x)w_L(t,x)}dxdt\right|\\
=&\left|\int_0^T\int_{\mathbb{T}}g(x)^2\sum_{|\Phi(hk_1)|\leq h^2M_0^2}\sum_{|\Phi(hk_2)|>h^2M_0^2}\ov{\widehat{w_0}(k_1)}\widehat{w_0}(k_2)e^{-i(k_1-k_2)x}e^{\frac{it}{h^3}(\Phi(hk_2)-\Phi(hk_1))}dxdt\right|\\
=&\left|\sum_{|\Phi(hk_1)|\leq h^2M_0^2}\sum_{|\Phi(hk_2)|>h^2M_0^2}2\pi\widehat{G}(k_1-k_2)\ov{\widehat{w_0}(k_1)}\widehat{w_0}(k_2)h^3\frac{e^{\frac{iT}{h^3}(\Phi(hk_2)-\Phi(hk_1))}-1}{\Phi(hk_2)-\Phi(hk_1)}\right|\\
\leq & Ch^3\sum_{|\Phi(hk_1)|\leq h^2M_0^2}\sum_{|\Phi(hk_2)|>h^2M_0^2}|\widehat{G}(k_1-k_2)\widehat{w_0}(k_1)\widehat{w_0}(k_2)|\frac{1}{\Phi(hk_2)-\Phi(hk_1)}
\end{split}
\end{equation*}
with $G(x)=g(x)^2$. 
  If $(k_1-r_h)(k_2-r_h)\geq 0$,  we have
\begin{equation*}
\begin{split}
|\Phi(hk_2)-\Phi(hk_1)|=\left|\int_{hk_1}^{hk_2}\Phi'(t)dt\right|
\geq h|\Phi'(hk_1)||k_2-k_1|
\geq ch^2|k_2-k_1|,
\end{split}
\end{equation*} 
in the case $|k_1-r_h|\geq 1$.
If otherwise $|k_1-r_h|<1$, we directly estimate 
\begin{equation*}
\begin{split}
|\Phi(hk_2)|-|\Phi(hk_1)|\geq h^2\left(M_0^2-\sup_{\xi\in[-2h,2h]}|\Phi'(\xi)|\right)>\frac{h^2M_0^2}{2}
\end{split}
\end{equation*}
by taking $n_0$ reasonable.  
  There are two possibilities in the case of  $(k_1-r_h)(k_2-r_h)<0$: either $\displaystyle{k_2\geq n_0+1,\frac{\xi^*}{h}\leq k_1<r_h}$, or $\displaystyle{k_2<\left\lfloor\frac{\xi^*}{h}\right\rfloor, r_h<k_1\leq n_0}$. For the first case, we have
\begin{equation*}
\begin{split}
\Phi(hk_2)-\Phi(hk_1)\geq & \Phi(hk_2)-\Phi(\xi^*)\\=&\frac{\Phi''(\sigma_h)h^2}{2}\left(
k_2-\frac{\xi^*}{h}\right)\left(
k_2+\frac{\xi^*}{h}-2r_h\right)+O(h^3M_0^3)\\
\geq & \frac{\Phi''(\sigma_h)h^2}{2}(k_2-k_1)(k_2+m_0+\mu_1-2r_h)+O(h^3M_0^3)\\
\geq & \frac{\Phi''(\sigma_h)h^2}{2}|k_2-k_1|(n_0+1+m_0+\mu_1-2r_h)+O(h^3M_0^2)\\
\geq &\frac{|k_2-k_1|h^2}{16}
\end{split}
\end{equation*}
 by choosing $h$ small enough, thanks to  \eqref{relation}, \eqref{modulo} and \eqref{relation}. In the case that $\displaystyle{k_2<\left\lfloor\frac{\xi^*}{h}\right\rfloor}$, we have
\begin{equation*}
\begin{split}
\Phi(hk_2)-\Phi(hk_1)\geq &\frac{\Phi''(\sigma_h)h^2}{2}(k_1-k_2)(-k_2-k_1+2r_h)+O(h^3M_0^3)\\
\geq &h^2|k_2-k_1|(-m_0-n_0+2r_h)+O(h^3M_0^2)\\
\geq &\frac{|k_1-k_2|h^2}{16}.
\end{split}
\end{equation*}
This implies that
\begin{equation*}
\begin{split}
\left|\int_0^T\int_{\mathbb{T}}g(x)^2w_H(t,x)\cdot\ov{w_L}(t,x)dxdt\right|
\leq &Ch\sum_{k_\neq k_2}|\widehat{G}(k_1-k_2)||\widehat{w_0}(k_1)||\widehat{w_0}(k_2)|\\
\leq &Ch\left(\sum_{k\in\mathbb{Z}}|\widehat{G}(k)|\right)\|w_0\|^2,
\end{split}
\end{equation*}
where we have used Young's convolution inequality. From this, we could improve the estimate of $\|w_L(0)\|^2$ as follows.
\begin{equation*}
\begin{split}
\|w_{L}(0)\|^2\leq &\kappa(n_0+2)T\int_0^T\int_{\mathbb{T}}|gw_{L}(t,x)|^2dxdt\\
=&\kappa(n_0+2)T\int_0^T\int_{\mathbb{T}}|g(x)w(t,x)|^2dxdt-\kappa(n_0+2)T\int_0^T\int_{\mathbb{T}}|g(x)w_{H}(t,x)|^2dxdt\\
-&2\kappa(n_0+2)T\Re\int_0^T\int_{\mathbb{T}}g(x)w_H(t,x)\cdot\ov{g(x)w_{L}(t,x)}dxdt\\
\leq & \kappa(n_0+2)T\int_0^T\|gw(t)\|^2dt+Ch\kappa(n_0+2)T\|w(0)\|^2,
\end{split}
\end{equation*} 
and
\begin{equation*}
\begin{split}
\|w(0)\|^2=&\|w_L(0)\|^2+\|w_H(0)\|^2\\
\leq & \left(C_T+\kappa(n_0+2)T\right)\int_0^T\|gw(t)\|^2dt+Ch\kappa(n_0+2)\|w(0)\|^2.
\end{split}
\end{equation*} 
The last term on the right hand side can be absorbed to the left, and this completes the proof.
\end{proof}

\appendix
\section{On the observability of fractional linear KP I}
In this appendix, we will give a proof of Theorem \ref{dichotomy} for fractional KP I equation 
\begin{equation}\label{fKPI}
\partial_tu-|D_x|^{\alpha}\partial_xu-\partial_x^{-1}\partial_y^2u=0.
\end{equation}

When $\alpha\geq 1$， the proof of observability can be reduced to the 1D uniform observability of 
$$ h^{\alpha+1}D_tv-\Phi_{\alpha}(hD_x)v=0,v=\chi_{\delta}(hD_x)v,
$$
with $\displaystyle{\Phi_{\delta}(\xi)=\phi_{\alpha}\left(\xi+h\left\lfloor\frac{\xi_0}{h}\right\rfloor\right), \phi_{\alpha}(\xi)=|\xi|^{\alpha}\xi+\frac{1}{\xi}}$, after doing the same reduction as in the beginning of section 3.2. Thus it would be sufficient to prove Proposition \ref{final} for solutions of \eqref{fKPI}. Actually, the proof of Proposition \ref{final} works also in the case $\alpha>1$. For $\alpha=1$, we need a little more argument.

Taylor expansion gives
$$ \Phi_1(\xi)=\frac{\Phi_1''(\sigma_h)}{2}(\xi-\sigma_h)^2+\frac{\Phi_1'''(\theta)}{6}(\xi-\sigma_h)^3.
$$
Note that $\Phi_1''(\sigma_h)=\phi_1''(\xi_0)=2A_0$ is independent of $h$, and we have
$$ \Phi_1(hD_x)=h^2A_0(D_x-r_h)^2+O(\|\Phi_1'''\|_{L^{\infty}}))((hD_x-\sigma_h)^3).
$$
with $r_h=\frac{\sigma_h}{h}\in[0,1)$. For $0<\delta\ll 1$ , we decompose
$$ v=v_1+v_2,\quad v_1=\chi_{A\delta}(h^{1/3}D_x)\chi_{\delta}(hD_x)v.
$$
Then $v_1$ solves
$$ D_tv_1-A_0(D_x-\sigma_h)^2v_1=O_{L^2\rightarrow L^2}(A\delta)v_1.
$$
We denote by $S_{\sigma_h}(t)$ the semi-group associated with the evolution Schr\"odinger operator $D_t-A_0(D_x-\sigma_h)^2$. From observability for classical Schr\"odinger equation, we have
$$ \|v_1(0)\|^2\leq C_T\int_0^T\|gS_{\sigma_h}(t)v_1(0)\|^2dt
$$ 
with constant $C_T$ independent of $\sigma_h\in (0,1)$. For this independence assertion, we refer to Lemma 2.4 of \cite{Burq-Zworski}. Therefore, we have from Duhamel formula that 
\begin{equation}
\begin{split}
\|v_1(0)\|^2\leq &C_T\int_0^T\left\| gv_1(t)-g\int_0^tO_{L^2\rightarrow L^2}(\delta)v_1(t')dt'\right\|^2dt \\
\leq &C_T\int_0^T\|gv_1(t)\|^2dt+AC_T\delta\|v_1\|_{T}^2\\
\leq &C_T\int_0^T\|gv_1(t)\|^2dt+AC_T\delta\|v_1(0)\|^2,
\end{split}
\end{equation}
where we have used the conservation of $L^2$ norm in the last step. For given $T>0$, we take $\delta>0$ sufficiently small in a priori, and thus
$$ \|v_1(0)\|^2\leq C_T\int_0^T\|gv_1(t)\|^2dt.
$$
The estimate of $v_2$ follows in the same way as in the proof of Lemma \ref{highfrequencies}. Therefore we have
$$ \|v_2(0)\|^2\leq C_T\int_0^T\|gv_2(t)\|^2dt.
$$ 
Finally, from the commutator estimate $\displaystyle{\left\|\left[\chi_{A\delta}(h^{1/3}D_x)\chi_{\delta}(hD_x),g\right]\right\|_{L^2\rightarrow L^2}\leq Ch^{1/3}},$ the proof is complete.

We now construct the conterexample of observability for the case $\alpha<1$. The construction is in the same spirit as in \cite{Rivas-Sun}.
\begin{proposition}\label{1Dcounterexample}
Suppose $0<\alpha<1$. Then for any $T>0$, there exists a sequence $v_n$, solutions of
$$ h_n^{1+\alpha}D_tv_n+\Phi_1(h_nD_x)v_n=0,
$$
such that
$$ \lim_{n\rightarrow\infty}\frac{\int_0^T\int_{\omega}|v_n(t,x)|^2dxdt}{\int_{\mathbb{T}}|v_n(0,x)|^2dx}=0.
$$
\end{proposition}
\begin{proof}
We may assume that $\omega=(-\pi,-\beta)\cup (\beta,\pi]$. Take $G(x)=e^{-\frac{x^2}{2}}$ and define $G^{\epsilon_n}(x)=\frac{1}{\sqrt{\epsilon_n}}G\left(\frac{x}{\epsilon_n}\right)$. Denote the Fourier coefficient of $G^{\epsilon_n}$ by $$g^{\epsilon_n}(k)=\frac{1}{2\pi}\int_{-\pi}^{\pi}G^{\epsilon_n}(x)e^{-ikx}dx=\frac{\sqrt{\epsilon_n}}{2\pi}\int_{-\frac{\pi}{\epsilon_n}}^{\frac{\pi}{\epsilon_n}}G(z)e^{-i\epsilon_nkz}dz.$$
The coefficient function $g^{\epsilon_n}(z)$ satisfies the following estimates:
\begin{equation}\label{coefficient}
\|g^{\epsilon_n}\|_{L^{\infty}(\mathbb{R})}=O(\epsilon_n^{1/2}), \|(g^{\epsilon_n})'\|_{L^{\infty}(\mathbb{R})}=O(\epsilon_n^{3/2}), \|(g^{\epsilon_n})''\|_{L^{\infty}(\mathbb{R})}=O(\epsilon_n^{5/2}).
\end{equation}

Take an even cut-off function $\psi\in C_c^{\infty}(\mathbb{R})$ with supp $\psi\subset [-B,B]$ and $0\leq \psi\leq 1$, $\psi|_{[-b,b]}\equiv 1$. We define
$$ v_{n,0}=\sum_{k\in\mathbb{Z}}g^{\epsilon_n}(k)\psi(\widetilde{h}_nk)e^{ikx},
$$
with $\widetilde{h_n}=h_n^{1-\alpha}$. The corresponding solution  is given explicitly by
$$ v_n(t,x)=\sum_{k\in\mathbb{Z}}g^{\epsilon_n}(k)\psi(\widetilde{h}_nk)\exp\left(ikx-\frac{it\Phi_1(\widetilde{h}_n^{\frac{1}{1-\alpha}}k)}{\widetilde{h}_n^{\frac{1+\alpha}{1-\alpha}}}\right).
$$
We first estimate the lower bound of the mass of initial data. 
$$ \|G^{\epsilon_n}\|_{L^2(\mathbb{T})}^2=\sum_{k\in\mathbb{Z}}|g^{\epsilon_n}(k)|^2\sim 1
$$
holds from Plancherel theorem and the definition of $g^{\epsilon_n}(k)$. We next estimate the mass away from the frequency scale $h_n^{-1}$, that is 
\begin{equation*}
\begin{split}
\sum_{k\in\mathbb{Z}}\left|(1-\psi(\widetilde{h}_nk))g^{\epsilon_n}(k)\right|^2\leq &\sum_{|\widetilde{h_n}k|>b}|g^{\epsilon_n}(k)|^2\\
\leq &\sum_{|\widetilde{h_n}k|>b}\frac{\epsilon_n}{4\pi^2}\left|\int_{\mathbb{R}}G(z)e^{-ik\epsilon_n}zdz\right|^2\\
=&\sum_{|\widetilde{h_n}k|>b}\frac{\epsilon_n}{4\pi^2}\left|\int_{\mathbb{R}}G(z)\frac{1}{-ik\epsilon_n}\frac{d}{dz}e^{-ik\epsilon_n}zdz\right|^2\\
\leq& \sum_{|\widetilde{h_n}k|>b}\frac{1}{4k^2\pi^2\epsilon_n}\|G'\|_{L^1(\mathbb{R})}^2.
\end{split}
\end{equation*}
By setting $\epsilon_n=\sqrt{\widetilde{h}_n}\ll 1$, we have $\|(1-\psi(\widetilde{h}_nD_x))G^{\epsilon_n}\|_{L^2(\mathbb{T})}\ll 1$ and then $\|v_{n,0}\|_{L^2(\mathbb{T})}\sim 1$.

Now we choose $B>0$ so that $|x-\|\Phi_1''\|_{L^{\infty}([-\delta,\delta])}Bt|\geq 2c_0>0$ mod $2\pi$ for all $x\in\omega=(-\pi,-\beta)\cup (\beta,\pi)$ and $|t|\leq T$.
Write
$$ v_n(t,x)=\sum_{k\in\mathbb{Z}}K_{t,x}^{(n)}(k)
$$
with
$$ K_{t,x}^{(n)}(z)=g^{\epsilon_n}(z)\psi(\widetilde{h}_nz)\exp\left(izx-i\widetilde{h}_n^{-\frac{1+\alpha}{1-\alpha}}\Phi_1(\widetilde{h}_n^{\frac{1}{1-\alpha}}z)t\right).
$$
From Poisson summation formula, we have
$$ v_n(t,x)=\sum_{m\in\mathbb{Z}}\widehat{K_{t,x}^{(n)}}(2\pi m).
$$
For fixed $m\in\mathbb{Z}$, 
\begin{equation*}
\begin{split}
\widehat{K_{t,x}^{(n)}}(2\pi m)=&\int_{\mathbb{R}}g^{\epsilon_n}(z)\psi(h_nz)e^{i\varphi_{t,x}(z)}dz\\
=&\int_{\mathbb{R}}g^{\epsilon_n}(z)\psi(h_nz)\mathcal{L}^2(e^{i\varphi_{t,x}(z)})dz
\end{split}
\end{equation*}
with $\displaystyle{\mathcal{L}=\frac{1}{i\varphi'_{t,x}(z)}\frac{d}{dz}}$ and $\varphi_{t,x}(z)=(x-2\pi m)z-\widetilde{h}_n^{-\frac{1+\alpha}{1-\alpha}}\Phi_1(\widetilde{h}_n^{\frac{1}{1-\alpha}}z)t$. Thus
\begin{equation*}
\begin{split}
\widehat{K_{t,x}^{(n)}}(2\pi m)=\int_{\mathbb{R}}\frac{d}{dz}\left(\frac{1}{i\varphi'_{t,x}(z)}\frac{d}{dz}\left(\frac{g^{\epsilon_n}(z)\psi(\widetilde{h}_nz)}{i\varphi'_{t,x}(z)}\right)\right)e^{i\varphi_{t,x}(z)}dz.
\end{split}
\end{equation*}
After tedious calculation, we have
\begin{equation*}
\begin{split}
&\frac{d}{dz}\left(\frac{1}{i\varphi'_{t,x}(z)}\frac{d}{dz}\left(\frac{g^{\epsilon_n}(z)\psi(\widetilde{h}_nz)}{i\varphi'_{t,x}(z)}\right)\right)\\ =&\frac{(g^{\epsilon_n})''\psi(\widetilde{h}_nz)+2\widetilde{h}_n(g^{\epsilon_n})'\psi'(\widetilde{h}_nz)+\widetilde{h}_n^2\psi''(\widetilde{h}_nz)g^{\epsilon_n}}{(\varphi'_{t,x})^2}\\
-&\frac{3((g^{\epsilon_n})'\psi(\widetilde{h}_nz)+\widetilde{h}_n\psi'(\widetilde{h}_nz)g^{\epsilon_n})\varphi''_{t,x}}{(\varphi'_{t,x})^3}\\
-&\frac{3g^{\epsilon_n}\psi(\widetilde{h}_nz)(\varphi''_{t,x})^2}{(\varphi'_{t,x})^4}.
\end{split}
\end{equation*}
From \eqref{coefficient}, we have 
$$ |\widehat{K_{t,x}^{(n)}}(2\pi m)|\leq \sup_{|\widetilde{h}_nz|\leq B} \frac{C\epsilon_n^{1/2}\|\psi\|_{W^{2,1}(\mathbb{R})}}{\left|(x-2\pi m)-\widetilde{h}_n^{\frac{-\alpha}{1-\alpha}}\Phi_1'(\widetilde{h}_n^{\frac{1}{1-\alpha}}z)t\right|^2}.
$$
Note that from Taylor expansion,
$$ \widetilde{h}_n^{\frac{-\alpha}{1-\alpha}}\Phi_1'(\widetilde{h}_n^{\frac{1}{1-\alpha}}z)t=\Phi_1''(\theta_n)\widetilde{h}_nzt-\Phi_1''(\theta_n)\sigma_{h_n}\widetilde{h}_n^{\frac{-\alpha}{1-\alpha}}t
$$
for some $\theta_n\in (\sigma_{h_n},\widetilde{h}_n^{\frac{1}{1-\alpha}}z)$. Therefore, for sufficiently large $n$, and for any $x\in 2\pi \mathbb{Z}+(-\beta,-\alpha)\cup (\beta,\pi]$, $$\left|x-\widetilde{h}_n^{\frac{-\alpha}{1-\alpha}}\Phi_1'(\widetilde{h}_n^{\frac{1}{1-\alpha}}z)t\right|\geq c_0>0 \textrm{ module } 2\pi, $$ thus
\begin{equation*}
\begin{split}
\sum_{m\in\mathbb{Z}} |\widehat{K_{t,x}^{(n)}}(2\pi m)|\leq & C\sum_{m\in\mathbb{Z}}\frac{C\epsilon_n^{1/2}}{|c_0-2\pi(m-p)|^2}\\
\leq & C\epsilon_n^{1/2}
\end{split}
\end{equation*}
holds for any $p\in\mathbb{Z}$.
Therefore,
$$ \int_0^T\int_{\omega}|v_n(t,x)|^2dxdt\leq C\epsilon_n^{1/2}T|\omega|\rightarrow 0, \textrm{ as }n\rightarrow\infty.
$$
\end{proof}

\begin{corollary}
	Suppose $0<\alpha<1$, then for any $T>0$, the observability for $u_n$, solutions of \eqref{fKPI} does not hold true.
\end{corollary}
\begin{proof}
We take $h_n,v_n$ as in Proposition \ref{1Dcounterexample}. Define
$$ u_n(t,x,y)=w_n(t,x)\exp\left(iyh_n^{-\frac{\alpha+2}{2}}\right),
$$
where $(h_n^{-\frac{\alpha+2}{2}})$ is a sequence of positive integers which converges to infinity. $u_n$ solves \eqref{fKPI} means that
$$ h_n^{\alpha+1}\partial_tw_n-|h_nD_x|^{\alpha}h_n\partial_xw_n-h_n^{-1}\partial_x^{-1}w_n=0.
$$
Now we set $w_n=v_ne^{\left\lfloor\frac{i\xi_0}{h_n}\right\rfloor x}$, and we have from Proposition \ref{1Dcounterexample} that
$$ \lim_{n\rightarrow\infty}\frac{\int_0^T\int_{\mathbb{T}^2}|g(x)u_n(t,x,y)|^2dxdydt}{\int_{\mathbb{T}^2}|u_n(0,x,y)|^2dxdy}=0.
$$
We finally need replace $\displaystyle{\int_0^T\int_{\mathbb{T}^2}|g(x)u_n(t,x,y)|^2dxdydt}$ by $\displaystyle{\int_0^T\int_{\mathbb{T}^2}|\mathcal{G}u_n(t,x,y)|^2dxdydt}$. This is guranteed by 
$$ \int_{\mathbb{T}}g(x')u_n(t,x',y)dx'\rightarrow 0, \quad \textrm{ in } L^2((0,T)\times\mathbb{T}_y)
$$
from our construction. This completes the proof.
\end{proof}

\subsection*{Acknowledgement} The author would like to thank his PhD advisor, Professor Gilles Lebeau, for his valuable suggestions and comments.


\begin{center} 
	
\end{center}

\end{document}